\documentclass[11pt, reqno, lettersize]{amsart}
\numberwithin{equation}{section}

\usepackage{amsmath,amsthm,amsfonts,amssymb,graphicx}

\theoremstyle{plain}
\newtheorem{thm}{Theorem}[section]
\newtheorem{lemme}[thm]{Lemma}
\newtheorem{prop}[thm]{Proposition}
\newtheorem{cor}[thm]{Corollary}
\newtheorem{defi}{Definition}

\theoremstyle{remark}
\newtheorem{rque}{Remark}

\theoremstyle{definition}

\newcommand{\rond}{\mathcal}
\newcommand{\id}{ \mathfrak}

\newcommand{\pt}{\cdot}
\newcommand{\pts}{\ldots}

\newcommand{\Z}{\mathbb{Z}}
\newcommand{\Q}{\mathbb{Q}}
\newcommand{\R}{\mathbb{R}}

\newcommand{\C}{\mathbb{C}}
\newcommand{\ec}{\textnormal}
\newcommand{\tq}{\vert}
\newcommand{\congru}{\equiv}

\newcommand{\barre}{\overline }
\newcommand{\norme}{\Vert}

\renewcommand{\mod}[1]{{\ifmmode\text{\rm\ (mod~$#1$)}\else\discretionary{}{}{\hbox{ }}\rm(mod~$#1$)\fi}}

\newcommand{\sumstar}{\sideset{}{^*}\sum}
\newcommand{\sumflat}{\sideset{}{^\natural}\sum}

\begin{document}

\title{A quadratic large sieve inequality over number fields}

\author{Leo Goldmakher}
\address{Department of Mathematics, University of Toronto, Toronto, ON, Canada} \email{leo.goldmakher@utoronto.ca}
\author{Beno\^it Louvel}
\address{Mathematisches Institut G\"ottingen \\
Bunsenstr. 3-5 \\ 37073 G\"ottingen, Germany}
\email{blouvel@uni-math.gwdg.de}

\thanks{First author partially supported by an NSERC grant. Second author supported by the Volkswagen Fundation.}
\keywords{}
\begin{abstract} We formulate and prove a large sieve inequality for quadratic characters over a number field. To do this, we introduce the notion of an $n$-th order Hecke family. We develop the basic theory of these Hecke families, including versions of the Poisson summation formula.
\end{abstract}
\subjclass[2010]{}
\maketitle
\setcounter{page}{1}

\section{Introduction}\label{sec:intro}

In \cite{HB}, Heath-Brown proved a large sieve inequality for quadratic characters:
\begin{equation}\label{intro:eq:1}
\sumflat_{M<a\le 2M} \Big\vert \sumflat_{N<b\le 2N} \lambda_b \left(\frac{a}{b}\right)\Big\vert^2 \ll (MN)^\varepsilon (M+N) \sumflat_{N< b\le 2N} \tq\lambda_b\tq^2 .
\end{equation}
Here $(\lambda_b)$ is any sequence of complex numbers, $\varepsilon > 0$, $M, N \geq 1$, $(\pt/\pt)$ is the Jacobi symbol, and the sums are restricted to odd squarefree integers. This bound has proved to be extremely useful in applications, and one might wish to generalize it. Heath-Brown \cite{HB1} has proved an analogue of (\ref{intro:eq:1}) with cubic characters, and in joint work with Blomer \cite{BGL}, the authors have proved an analogue for characters of arbitrary order.\\

Our goal in the present work is to generalize (\ref{intro:eq:1}) in a different direction, by extending it to number fields. The only such generalization we are aware of is a recent result of Onodera \cite{O}, who proves a quadratic large sieve for $\Q(i)$. The apparent neglect of this natural problem can be largely attributed to  the difficulty of formulating the proper number field generalization. More precisely, it is not clear what an appropriate analogue of the Jacobi symbol is over a number field. The most obvious candidate, the power residue symbol, is not suitable: any analogue of (\ref{intro:eq:1}) requires a symbol admitting integral ideal inputs, while the power residue symbol $(a / \id{b})$ is defined for number field elements $a$ and ideals $\id{b}$ and it is not obvious how to extend the top entry to ideals. (We give a brief description of the power residue symbol in Section \ref{sec:app2}.) To get around this, we introduce the notion of an \emph{$n$-th order Hecke family}. For a number field $k$ with ring of integers $\mathcal{O}$, let $I(\id{a})$ denote the set of integral ideals coprime to a given integral ideal $\id{a}$, and let $\rond{N}(\id{a})$ be the absolute norm of $\id{a}$. Given a Hecke character $\chi$ modulo $\id{f}$, the infinite type of $\chi$ is defined to be the unique character $\chi_\infty$ on $k\otimes_\Q\R$ satisfying $\chi\big((x)\big) = \chi_\infty(x)$ for every $x\in \mathcal{O}$ with $x\congru 1 \mod{\id{f}}$.

\begin{defi} 
\label{defi:HeckeFamily}
Given $n \geq 2$, let $k$ be a number field containing the group $\mu_n$ of $n$-th roots of unity. An \textbf{$n$-th order Hecke family} (with respect to a fixed ideal $\id{c}$) is a collection 
\[
\rond{F} = \rond{F}_\id{c} = 
	\{\chi_\id{a}:\id{a}\in I(\id{c}), \id{a} \ec{ squarefree}\}
\]
of primitive Hecke characters of trivial infinite type, satisfying the following three properties:
\begin{enumerate}

\item the order of $\chi_\id{a}$ divides $n$ for every character $\chi_\id{a} \in \rond{F}$\textup{;}

\item $\rond{F}$ satisfies a reciprocity law of the form: there exists a finite group $G$, a homomorphism $[\pt]$ from $I(\id{c})$ to $G$, and a map $C : G \times G \to \mu_n$ such that
        \begin{equation}\label{eq:Hecke:recip}
            \chi_\id{a}(\id{b}) = \chi_\id{b}(\id{a}) \ C([\id{a}],[\id{b}])
        \end{equation}
    for all coprime ideals $\id{a}, \id{b} \in I(\id{c})$\textup{;} and 
\item for all coprime ideals $\id{a}, \id{b} \in I(\id{c})$ satisfying $[\id{a}]=[\id{b}]$, $\chi_\id{a}\barre{\chi}_\id{b}$ is a primitive Hecke character modulo $\id{ab}$.
\end{enumerate}
\end{defi}

\begin{rque}
Property (3) generalizes the following property of the Jacobi symbol: if $a$ and $b$ are positive odd coprime  integers, then $(ab/\pt)$ is a Dirichlet character modulo $ab$ if $a \congru b \mod{4}$. 
This property will play an essential role in our argument; see \eqref{eq:defB3} and Section \ref{subsec:explicitformula-bis}. 
\end{rque}

\begin{rque}\label{rk2}
The ideal $\id{c}$ in an $n$-th order Hecke family plays the same role as the modulus 4 in quadratic reciprocity.
\end{rque}

\begin{rque}\label{rk}
Note that if $\{\chi_\id{a} : \id{a}\in I(\id{c})\}$ is an $n$-th order Hecke family, then the set $\{\chi_\id{a}^{n/d} : \id{a}\in I(\id{c})\}$ is a $d$-th order Hecke family for any non-trivial divisor $d$ of $n$.
\end{rque}

It is not clear \emph{a priori} that any such family exists. An example of a quadratic Hecke family (indeed, the motivating example) was constructed by Fisher and Friedberg in \cite{FF} -- see Section \ref{sec:app2} for a brief description of their work. Their construction is quite natural, and can be readily extended to produce Hecke families of any order; 
Remark \ref{rk} then indicates how to modify their construction to produce other Hecke families. It is an interesting question to determine whether all Hecke families are thus induced from the Fisher-Friedberg family.\\

With this notation in hand, we can now state our main result:
\begin{thm}\label{thm1}
Let $\id{c}$ be an integral ideal of a number field $k$, and let $\{\chi_\id{a}\}$ be a quadratic Hecke family with respect to $\id{c}$. Given any $\varepsilon > 0$, $M, N \geq 1$, and a sequence $(\lambda_\id{b})$ of complex numbers parametrized by integral ideals of $k$, we have
\[
\sumstar_{\rond{N}\id{a}\le M} \Big\vert \sumstar_{\rond{N}\id{b}\le N} \lambda_\id{b} \chi_\id{b}(\id{a})\Big\vert^2
\ll_{k,\id{c},\varepsilon}
(MN)^\varepsilon (M+N) \sumstar_{\rond{N}\id{b}\le N} \tq \lambda_\id{b}\tq^2 .
\]
Here and henceforth $\sumstar$indicates that the sum is restricted to squarefree ideals of $I(\id{c})$.
\end{thm}

\noindent This generalizes Heath-Brown's result \cite{HB} ($k=\Q$) and Onodera's result \cite{O} ($k=\Q(i)$). We record the following consequence, which plays an important role in \cite{BGL}:
\begin{cor}
Given $n\ge 3$ even, let $k$ be a number field containing the group $\mu_n$ of $n$-th roots of unity. 
If $\{\chi_\id{a}: \id{a}\in I(\id{c})\}$ is an $n$-th order Hecke family and $(\lambda_\id{b})$ is a sequence of complex numbers parametrized by integral ideals of $k$, we have
\[
\sumstar_{\rond{N}\id{a}\le M} \Big\vert \sumstar_{\rond{N}\id{b}\le N} \lambda_\id{b} \chi_\id{b}^{n/2}(\id{a})\Big\vert^2
\ll_{k,\id{c},\varepsilon}
(MN)^\varepsilon (M+N) \sumstar_{\rond{N}\id{b}\le N} \tq \lambda_\id{b}\tq^2
\]
for all $\varepsilon > 0$ and all $M,N \geq 1$.
\end{cor}

We end this introduction by giving a short overview of the organization of the paper. In Section~\ref{sec:app2} we describe an example of an $n$-th order Hecke family. In Section \ref{sec:app}, we develop some necessary summation formulas over number fields. Section \ref{sec:overview} is devoted to reducing Theorem \ref{thm1} to a recursive estimate, Theorem \ref{cor}. This theorem is then reduced further in Section \ref{sec:Reduction} to an upper bound, stated in Proposition \ref{prop}. This proposition is proved in the final two sections of the paper. 

\noindent\\
\textbf{Acknowledgments.}
The authors are grateful to Valentin Blomer for his encouragement, and to the anonymous referee for helpful suggestions. The first author would also like to thank the mathematics department at the University of G\"ottingen for their hospitality.

\section{The Fisher-Friedberg Hecke family}\label{sec:app2}

As discussed in the introduction, an example of a quadratic Hecke family was first given by a construction of Fisher and Friedberg \cite{FF}. This was later extended to Hecke families of all orders by Friedberg, Hoffstein, and Lieman in \cite{FHL}. In this section, we give a brief description of their work. \\

Let $n\ge 2$ be a fixed integer and let $k$ be a number field containing the group $\mu_n$ of $n$-th roots of unity. Let $\rond{O}$ be the ring of integers of $k$. For each place $v$ of $k$, let $k_v$ denote the completion of $k$ at $v$. For $v$ nonarchimedean, let $\id{p}_v$ denote the corresponding ideal of $\rond{O}$. For an integral ideal $\id{c}$ of $k$, we denote by $I(\id{c})$ the set of integral ideals of $k$ coprime to $\id{c}$ and by $I^*(\id{c})$ the group of fractional ideals of $k$ coprime to $\id{c}$; for a set $S$ of places of $k$, we define $I(S)$ and $I^*(S)$ analogously.\\

We first recall the definition of the $n$-th power residue symbol (see \cite[Chap.\ 19]{SD} or \cite[Exercises p. 348]{CF}). For $a\in k$, let $S_a$ be the set of places of $k$ which either divide $n$ or ramify in $k(a^{1/n})/k$. For $\id{p}\in I^*(S_a)$, let $F_a(\id{p})$ be the Frobenius automorphism corresponding to $\id{p}$. Extending this multiplicatively to all fractional ideals yields the Artin map $F_a: I^*(S_a) \to \ec{Gal}(k(a^{1/n})/k)$. For any prime ideal $\id{p}$ of $I^*(S_a)$, one has
\[
F_a(\id{p})(a^{1/n}) = (a/\id{p}) a^{1/n},
\]
for some $n$-th root of unity $(a/\id{p})$ which is independent of the choice of $a^{1/n}$. The symbol $(a/\id{p})$ is called an $n$-th power residue symbol because $(a/\id{p})=1$ is equivalent to $a$ being an $n$-th power in $k_v$ (where $v$ denotes the place corresponding to $\id{p}$). One can then extend this multiplicatively to a symbol $\chi_a(\id{b})=(a/\id{b})$ for any $\id{b}\in I^*(S_a)$. One of the properties of the power residue symbol is that $(a/\id{b})=1$ if $a\equiv 1 \mod {\id{b}}$. We refer to \cite[p. 348-350]{CF} for a more complete description of the power residue symbol.\\

Having described the $n$-th power residue symbol $\chi_a$ for a field element $a$, our next task is to extend this to a character $\chi_{\id{a}}$ with $\id{a}$ an ideal. This construction proceeds in several steps. After constructing an appropriate ideal $\id{c}$ (see Remark \ref{rk2} in the introduction), we generate a set $\rond{E}$ of integral ideals, which parametrizes the fractional ideals coprime to $\id{c}$ up to $n$-th power factors. One can then define a new symbol $\chi_{\id{a}}$ in terms of the parameter ideal in $\rond{E}$ corresponding to $\id{a}$. Finally, we show that the set of all such $\chi_{\id{a}}$ forms an $n$-th order Hecke family. We now carry out this construction in more detail.\\

Let $S$ be a finite set of places of $k$, containing all the archimedean places and the places dividing $n$, and large enough so that the ring $\rond{O}_S$ of $S$-integers has class number one. We construct an integral ideal $\id{c}$ of $k$ by setting $\id{c}=\prod \id{p}_v^{n_v}$, where $n_v$ is chosen to be $0$ if $v\not\in S$ or if $v\mid \infty$, $1$ if $v \nmid n$, and large enough that every $x\in k_v$ with $v(x)\ge n_v$ is an $n$-th power in $k_v$ if $v \mid n$. Note that for $v\mid n$, the integers $n_v$ have been determined explicitly by Hasse \cite[Property X, p. 46]{H}. If $n=2$ and $k$ has some real embedding, we write, for $x\in k^\times$, that $x>0$ if all real embeddings of $x$ in $\R$ are positive.\\

Let $H_\id{c}$ be the ray class group (narrow ray class group if $n=2$) modulo $\id{c}$, and let $R_\id{c}=H_\id{c}\otimes \Z/n\Z$.
Write the finite group $R_\id{c}$ as a product of cyclic groups, choose a generator for each, and let $\rond{E}_0$ be a set of ideals of $\rond{O}$ coprime to $\id{c}$ which represent these generators. For each $E_0\in \rond{E}_0$, choose $m_{E_0}\in k^\times$ such that $E_0\rond{O}_S=m_{E_0}\rond{O}_S$; up to multiplication by a unit of $k$, we may and do assume that $m_{E_0}>0$ (only relevant for $n=2$ and $k$ having real embedding). Let $\rond{E}$ be a full set of representatives for $R_\id{c}$ of the form $\prod E_0^{n_{E_0}}$, $n_{E_0} \geq 0$. If $E =\prod E_0^{n_{E_0}}$ is such a representative, set $m_E=\prod m_{E_0}^{n_{E_0}}$ (we have $m_E>0$ for all $E\in \rond{E}$). Without loss of generality, we suppose $\rond{O}\in\rond{E}$ and $m_\rond{O}=1$.\\

Let $\id{a},\id{b} \in I^*(\id{c})$ be coprime. Write $\id{a}=(x)E\id{g}^n$ with $x \congru 1 \mod{\id{c}}$, $x>0$, $E\in \rond{E}$ and $(\id{g},\id{b})=(1)$. Let $m_\id{a}=xm_E$; then, the $n$-th order residue symbol $(m_\id{a}/\id{b})$ is well defined (\cite[Lemma 1.1]{FF}). Accordingly, we define $(\id{a}/\id{b})=(m_\id{a}/\id{b})$ and the character $\chi_\id{a}$ by $\chi_\id{a}(\id{b})=(\id{a}/\id{b})$. Note that this construction of the characters $\chi_\id{a}$ is non-canonical, since it depends on all the choices made above. \\

These characters generalize the power residue symbol, in the sense that for $a\congru 1 \mod{\id{c}}$ and $a>0$, one has $\chi_{(a)}=\chi_a$. The most important property of the characters $\chi_\id{a}$ is that they satisfy the reciprocity law \eqref{eq:Hecke:recip} (see \cite[Lemma 3.2]{FF}) with $G=R_\id{c}$ and $[\pt]$ the projection from $I^*(\id{c})$ to $R_\id{c}$. We now show that for all $\id{a}\in I^*(\id{c})$, $\chi_\id{a}$ is a Hecke character of order $n$ modulo $\id{ca}$. (Although surely well-known to the experts, this does not seem to be mentioned anywhere in the literature.)
Consider first the case where $k$ is imaginary. Note that $\chi_\id{a}(\rond{O})=\chi_\rond{O}(\id{a})=1$ by definition. Let $x\in k$, $x\congru 1 \mod{\id{ca}}$. The class of $(x)$ in $R_\id{c}$ is trivial, thus by definition we have $\chi_{(x)}=\chi_x$. From the reciprocity law \eqref{eq:Hecke:recip}, we obtain

\begin{align*}
\chi_\id{a}\big((x)\big)
&= \chi_{(x)}(\id{a}) C([\id{a}],[(x)])\\
&= \chi_x(\id{a}) C([\id{a}],[\rond{O}])\\
&= \chi_x(\id{a}) \chi_\id{a}(\rond{O}) \barre{\chi}_\rond{O}(\id{a})\\
&= \left(\frac{x}{\id{a}}\right) =1.
\end{align*}

If $k$ has real embeddings, i.e. if $n=2$, the above proof does not work, since $x\congru 1 \mod{\id{c}}$ does not imply that the class of $(x)$ is trivial in $R_\id{c}$ (we do not know whether $x>0$). In this case, by definition of $\chi_\id{a}$ we have $\chi_\id{a}=\chi_{m_\id{a}}$, where $m_\id{a}\in k$ is defined by $m_\id{a}=xm_E$ and $\id{a}=(x) E \id{g}^n$ with $x \equiv 1 \mod{\id{c}}$ and $x > 0$. Since $m_E>0$ and $x>0$, we have $m_\id{a}>0$ for all $\id{a}\in I^*(\id{c})$. From the law of quadratic reciprocity, one  shows that the infinite type of $\chi_{m_\id{a}}$ is $\prod (\pt,m_\id{a})_v$, the product being taken over the real infinite places of $k$. Therefore, the character $\chi_\id{a}$ is a Hecke character of trivial infinite type. Moreover, it is easily seen that for all $\id{a},\id{b}\in I(\id{c})$,
\[
\chi_\id{a}\chi_\id{b}=\chi_\id{ab},
\]
where both sides are viewed as Hecke characters modulo $\id{cab}$.\\

We remark that for any prime ideal $\id{p} \in I(\id{c})$, the character $\chi_\id{p}\mod{\id{c}\id{p}}$ is not induced by a character modulo $\id{c}$. Thus $\chi_\id{p}$ is a Hecke character modulo $\id{cp}$ of conductor $\id{c}_\id{p}\id{p}$, for some ideal $\id{c}_\id{p}$ dividing $\id{c}$. From this and multiplicativity, we deduce that if $\id{f}_\id{a}$ denotes the conductor of $\chi_\id{a}$, then $\id{f}_\id{a}=\id{c}_\id{a}\id{a}_0$ for some ideal $\id{c}_\id{a}$ dividing $\id{c}$. Here $\id{a}_0$ is defined as the product of the prime ideals dividing the $n$-th power-free part of $\id{a}$.\\

Let us abuse notation and denote by $\chi_\id{a}$ the primitive character inducing $\chi_\id{a}$. Adopting the same convention for $\chi\psi$ (i.e.\ letting this represent the primitive character inducing the product of the two Hecke characters $\chi$ and $\psi$), one easily sees that the primitive Hecke character $\chi_\id{a}$ inherits the above properties. Moreover, given squarefree, coprime ideals $\id{a},\id{b}\in I(\id{c})$ which are in the same class in $R_\id{c}$, one can  write $\id{a}=(x_\id{a}) E_\id{a} \id{g}_\id{a}^n$ and $\id{b}=(x_\id{b})E_\id{b}\id{g}_\id{b}^n$. Then, for any $x\congru 1 \mod{\id{ab}}$, we have

\begin{align*}
\chi_\id{a}\barre{\chi}_\id{b}\big((x)\big)
= \left(\frac{x_\id{a}/x_\id{b}}{x}\right) \left(\frac{m_{E_\id{a}}/ m_{E_\id{b}}}{x}\right)
= \left(\frac{x_\id{a}/x_\id{b}}{x}\right),
\end{align*}
since $E_\id{a}=E_\id{b}$. Moreover, since $x_\id{a},x_\id{b} \congru 1 \mod{\id{c}}$, we have

\begin{align*}
\left(\frac{x_\id{a}/x_\id{b}}{x}\right)
=\left(\frac{x}{x_\id{a}/x_\id{b}}\right)
=\left(\frac{x}{\id{a}}\right) \left(\frac{x}{\id{b}}\right)
=1.
\end{align*}
This shows that the characters $\chi_\id{a}$ described explicitly above are an example of an $n$-th order Hecke family.

\section{A Poisson summation formula over number fields} \label{sec:app}

The goal of this section is to develop a number field version of the Poisson summation formula.\\

Associated to the number field $k$ we have the following parameters: $d=[k:\Q]$ is the degree of the extension; $r_1$ is the number of real places and $r_2$ the number of complex ones (so that $d=r_1+2r_2$); $A_k= (2^{r_1}\tq d_k\tq (2\pi)^{-d})^{1/2}$, where $d_k$ is the discriminant; and $\alpha_k$ is the residue at $s=1$ of the completed $L$-function $\Lambda(s) = A_k^s \Gamma(s/2)^{r_1} \Gamma(s)^{r_2} \zeta_k(s)$.\\

Any primitive ray class character $\chi \mod{\id{f}}$ over $k$ (i.e. a Hecke character of trivial infinite type) satisfies the functional equation

\begin{equation}\label{eq:Hecke:3}
\Lambda(\chi,s) = \epsilon(\chi) (\rond{N}\id{f})^{1/2-s} \Lambda(1-s,\barre{\chi}),
\end{equation}
where

\begin{equation}\label{eq:Hecke:4}
\Lambda(\chi,s) =  \left(\frac{2^{r_1}\tq d_k\tq}{(2\pi)^d}\right)^{s/2} \Gamma\left(\frac{s}{2}\right)^{r_1}{\Gamma(s)}^{r_2} L(\chi,s) .
\end{equation}
Moreover, if the character $\chi$ is quadratic, we know that $\epsilon(\chi)=1$.\\

Given $h:\R\to \R$ a smooth function with compact support, denote by $\hat{h}$ the Mellin transform of $h$. We define the transform $\dot{h}:\R\to \R$ by
\begin{equation}\label{eq:transfo1}
\dot{h}(x) = \int_0^\infty h(t) \rond{K}(tx) \, dt,
\end{equation}
where $\rond{K}:\R_{\geq 0} \to \C$ is the function given by
\begin{equation*}
\rond{K}(t) = \frac{A_k^{-1}}{2\pi i} \int_{(\sigma)} \left(\frac{\Gamma(s/2)}{\Gamma\big((1-s)/2\big)}\right)^{r_1}
\left(\frac{\Gamma(s)}{\Gamma(1-s)}\right)^{r_2} \left(\frac{t}{A_k^2}\right)^{-s} \,ds.
\end{equation*}
Note that the function $\rond{K}(t)$ depends only on the field $k$.
We also define the transform
\begin{equation}
\ddot{h}(x) = \int_0^\infty h(t^2) \rond{K}(tx) \, dt.
\end{equation}
The function $\dot{h}$ satisfies the property
\begin{equation*}
\dot{h}(x) \ll \tq x\tq^{-A}, \qquad \ec{for all $x\neq 0$, for all $A>0$},
\end{equation*}
as well as
\begin{equation}\label{eq:Parseval}
\int_\R h(x^2) \,dx =  \int_\R \dot{h}(x^2) \,dx.
\end{equation}
We have the following results:

\begin{lemme}\label{lemme:Poisson}
Let $X>0$ be given. Then

\begin{equation*}
\sum_{\id{b}\neq 0} h\left(\frac{\rond{N} \id{b}}{X}\right)
= \frac{\alpha_k}{A_k}X \hat{h}(1) - \delta_{d=2}h(0)\alpha_kA_k  + X \sum_{\id{b}\neq 0} \dot{h}\left(X \rond{N}(\id{b})\right) ,
\end{equation*}
where $d$ is the degree of $k / \Q$ and $\delta_{d=2}$ is 1 if $d = 2$ and 0 otherwise.
\end{lemme}

\begin{lemme}\label{lemme:Poisson-char}
Let $X>0$ be given. For a non-trivial primitive ray class character $\chi \mod{\id{f}}$, one has

\begin{equation*}
\sum_{\id{b}\in I(\id{f})} \chi(\id{b}) h\left(\frac{\rond{N} \id{b}}{X}\right)
= \frac{\epsilon(\chi) X}{\sqrt{\rond{N}(\id{f})}} \sum_{\id{b}\in I(\id{f})} \barre{\chi}(\id{b}) \dot{h}\left(\frac{X \rond{N}(\id{b)}}{\rond{N}(\id{f})}\right),
\end{equation*}
where $\epsilon(\chi)$ comes from the functional equation \eqref{eq:Hecke:3}. 
\end{lemme}

\begin{lemme}\label{lemme:Poisson-coprim}
Let $X>0$ be given.
Let $Y\le X\le Z$ and let $L\ge (Z/X)^2$. Then, for any $A$,

\begin{align*}
&\sum_{(\id{b},\id{m})=1} h\left(\frac{\rond{N}(\id{b})}{X}\right)
= \frac{\varphi(\id{m})}{\rond{N} \id{m}} \alpha_kA^{-1} X \hat{h}(1)
- \sum_{\substack{\id{d} \mid \id{m}\\ \rond{N}(\id{d})> Z}} \mu(\id{d}) \alpha_kA^{-1} \frac{X}{\rond{N}(\id{d})} \hat{h}(1)\\
&- \delta_{d=2} h(0) \alpha_kA  \sum_{\substack{\id{d} \mid \id{m}\\ \rond{N}(\id{d})\le Z}} \mu(\id{d})
+\sum_{\substack{\id{d} \mid \id{m}\\ Y<\rond{N}(\id{d})\le Z}} \mu(\id{d}) \frac{X}{\rond{N}(\id{d})}  \sum_{\substack{\id{a}\neq 0\\ \rond{N}(\id{a})\le L}} \dot{h}\left(\frac{X \rond{N}(\id{a})}{\rond{N} (\id{d})}\right)\\
& + \rond{O}\left(Z (Z/X)^{-A} \right)+\rond{O} \left(X(X/Y)^{-A}\right),
\end{align*}
where $\delta$ is defined as in Lemma \ref{lemme:Poisson} and $d$ is the degree of $k/\Q$.
\end{lemme}

\begin{proof}[Proof of Lemma \ref{lemme:Poisson}]
Define

\begin{equation}
\Lambda(s) = \left( \frac{2^{r_1}\tq d_k\tq}{(2\pi)^d}\right)^{s/2} \Gamma(s/2)^{r_1} \Gamma(s)^{r_2} \zeta_k(s).
\end{equation}
Then $\Lambda(s)$ can be analytically continued to a meromorphic function on the whole $s$-plane, and satisfies the functional equation

\begin{equation}
\Lambda(s)= \Lambda(1-s).
\end{equation}
Moreover,  the poles of $\Lambda(s)$ are simple and located at $s=0$ and $s=1$. Recall that $\alpha_k=\ec{Res}_{s=1}\Lambda(s)$ and that $A_k=(\tq d_k\tq (2\pi)^{-d})^{1/2}$. By the inverse Mellin transform, for any $\sigma > 1$ we have

\begin{align*}
\sum_\id{b} h(\rond{N} \id{b})
&=  \frac{1}{2\pi i} \int_{(\sigma)} \hat{h}(s) \zeta_k(s) \,ds\\
&=  \frac{1}{2\pi i}\int_{(\sigma)} \hat{h}(s) \Gamma(s/2)^{-r_1}\Gamma(s)^{-r_2} A_k^{-s} \Lambda(s)  \,ds.
\end{align*}
Let $-1<\sigma'\leqslant 0$. The only possible poles of the integrand are located at $s=1$ and at $s=0$. At $s=1$, a simple pole occurs  with residue $\beta = \alpha_k \hat{h}(1)/A_k$. We know that $\Lambda(s) \Gamma(s)^{-d/2} A_k^{-s}=\zeta(s)$ is entire at $s=0$; actually it has a zero of order $r_2-1$. Moreover, $\hat{h}(s)$ may have a simple pole at $s=0$, of residue $h(0)$. Thus the integrand has no pole at $s=0$ if $d>2$ and a simple pole of residue $h(0) \zeta_k(0)=-h(0)\alpha_k A_k$ if $d=2$.
Moving the line of integration to $\Re(s)=\sigma'$ and applying the functional equation, we obtain

\begin{align*}
&\sum_{\id{b}\neq 0} h(\rond{N} \id{b})\\
&=  \beta  - \delta_{d=2} h(0)\alpha_kA_k + \frac{1}{2\pi i}\int_{(\sigma')} \hat{h}(s) \Gamma\left(\frac{s}{2}\right)^{-r_1}\Gamma(s)^{-r_2} A_k^{-s} \Lambda(s)  \,ds\\
&=  \beta  - \delta_{d=2} h(0)\alpha_kA_k+\frac{1}{2\pi i} \int_{(1-\sigma')} \hat{h}(1-s) \left(\frac{\Gamma(s/2)}{\Gamma\big((1-s)/2\big)}\right)^{r_1} \left(\frac{\Gamma(s)}{\Gamma(1-s)}\right)^{r_2} A_k^{2s-1} \zeta_k(s) \,ds.
\end{align*}
Expanding the $L$-function, exchanging the sums and integrals, and applying an inverse Mellin transform, we find
\[
\sum_{\id{b}\neq 0} h(\rond{N} \id{b})
= \beta  - \delta_{d=2} h(0)\alpha_kA_k + \sum_{\id{b}\neq 0}  \dot{h}\left(\rond{N}\id{b}\right).
\]
We deduce that
\[
\sum_{\id{b}\neq 0} h\left(\frac{\rond{N} \id{b}}{X}\right)
= \frac{\alpha_k}{A_k} X \hat{h}(1) - \delta_{d=2} h(0)\alpha_kA_k  + X \sum_\id{b} \dot{h}\left(X \rond{N}(\id{b})\right).
\]
\end{proof}

Lemma \ref{lemme:Poisson-char} is proved analogously to Lemma \ref{lemme:Poisson}, and Lemma \ref{lemme:Poisson-coprim} can be deduced from Lemma \ref{lemme:Poisson} as in \cite[Lemma 13]{HB}.

\section{Overview of the proof of Theorem \ref{thm1}}\label{sec:overview}

We are now ready to proceed to the first reduction in the proof of Theorem \ref{thm1}. As usual with large sieves, we first renormalize the sum under consideration. Given $\{\chi_\id{a}\}$ a quadratic Hecke family and $M, N \geq 1$, set

\begin{equation}\label{eq:defB1}
B_1(M,N) = \sup_{\norme \lambda \norme = 1} \sumstar_{\id{a} \sim M} \Big\vert \ \sumstar_{\id{b}} \lambda_\id{b} \chi_\id{b}(\id{a})\Big\vert^2,
\end{equation}
where $\id{a} \sim M$ means $M < \rond{N}\id{a} \le 2M$, the supremum is taken over all sequences $\lambda = (\lambda_\id{b})$ of support $N$ (i.e.\ $\lambda_\id{b} = 0$ whenever $\id{b} \not\sim N$ or $\id{b}$ is not squarefree), and $\norme \lambda \norme$ is defined by
\[
\norme \lambda\norme^2=
\sumstar_\id{b} \tq\lambda_\id{b}\tq^2 .
\]
Theorem \ref{thm1} is equivalent to 
\begin{equation}\label{eq:B1Bound}
B_1(M,N) \ll (MN)^\varepsilon (M+N)
\end{equation}
(here and throughout, the implicit constant is allowed to depend on $k$, $\id{c}$, $\varepsilon$, and nothing else).
Rather than proving \eqref{eq:B1Bound} directly, we derive it from a sequence of weaker estimates of the form
\begin{equation*}\tag{$E_\alpha$}
B_1(M,N)\ll (MN)^\varepsilon (M+N^\alpha), \qquad \ec{for all } M,N \geq 1 \ec{ and } \varepsilon>0.
\end{equation*}
We will show that, for $\alpha\ge1$, the bound $(E_\alpha)$ is self-improving:
\begin{thm}\label{thm2}
For every $\alpha \geq 1$, the upper bound $(E_\alpha)$ implies the upper bound $(E_{2-1/\alpha})$.
\end{thm}
Thus, to prove Theorem \ref{thm1} it suffices to prove that $(E_\alpha)$ holds for some $\alpha > 1$; iterating Theorem \ref{thm2} yields the bound $(E_1)$, which is equivalent to Theorem \ref{thm1}. Directly following the proof of Lemma \ref{lemme:duality} below, we will show
\begin{lemme}\label{lemme:initial-bound}
The bound $(E_2)$ holds.
\end{lemme}

The proof of Theorem \ref{thm2} is quite involved, and we attack it in several steps. First, we reduce it to the following estimate.
\begin{thm}\label{cor}
Let $M,N\ge 1$, $\varepsilon>0$ and suppose $(E_\alpha)$ holds. 
Then
\[
B_1(M,N) \ll (MN)^\varepsilon \left(M+N+N^{2\alpha-1}M^{1-\alpha}\right).
\]
\end{thm}
\noindent 
Next, in Section \ref{sec:Reduction}, we reduce Theorem \ref{cor} to a bound on a related quantity $B_3$ (see Proposition \ref{prop}). In Section \ref{sec:poisson} we apply a Poisson summation formula to $B_3$ to obtain an explicit formula (\ref{eq:explicit}). Finally, in Section \ref{sec:proof} we estimate each term of the explicit formula individually and conclude the proof of Proposition \ref{prop}. At the heart of our proof is a cancellation between the two main terms of this explicit formula. It is in the analysis of this explicit formula that our argument diverges most radically from that of \cite{BGL}. \\

The rest of this section is devoted to proving Lemma \ref{lemme:initial-bound} and deducing Theorem \ref{thm2}  from Theorem \ref{cor}. We begin with two standard and useful lemmas. First, we observe that $B_1$ is roughly an increasing function. The proof is substantially similar to that in \cite[Lemma 9]{HB}, to which we refer the reader for details.

\begin{lemme}\label{lemme:increasing}
There exists a positive constant $C$ such that if $M_2 \ge C M_1 \log(2M_1N)$ with ${M_1,M_2,N \ge 1}$, then
\begin{equation*}
B_1(M_1,N) \ll B_1(M_2,N).
\end{equation*}
\end{lemme}
\noindent Next, we show that $B_1$ is roughly symmetric in its arguments.
\begin{lemme}\label{lemme:duality}
For all $M, N \geq 1$, we have
$
B_1(M,N) \ll B_1(N,M).
$
\end{lemme}

\begin{proof}
Given a Hecke family $\chi_\id{a}(\id{b})$ and a sequence $\lambda = (\lambda_\id{a})$ parametrized by integral ideals, let $G$ denote the finite group with respect to which the reciprocity law \eqref{eq:Hecke:recip} holds. For each $g \in G$, we create a twisted sequence $\lambda^g = (\lambda_\id{a}^g)$ defined by
\[
\lambda_\id{a}^g = C(g,[\id{a}]) \lambda_\id{a}.
\]
The reciprocity law implies
\begin{align*}
\sumstar_\id{b}\Big\vert \sumstar_{\id{a}\sim M} \lambda_\id{a} \chi_\id{b}(\id{a})\Big\vert^2
&= \sum_{g\in G} \sumstar_{[\id{b}]=g} \Big\vert \sumstar_{\id{a}\sim M}  \lambda^g_\id{a} \chi_\id{a}(\id{b})\Big\vert^2 \\
&\le  \sum_{g\in G}\sumstar_\id{b} \Big\vert \sumstar_{\id{a}\sim M} \lambda^g_\id{a} \chi_\id{a}(\id{b}) \Big\vert^2 \\
&\le \sum_{g\in G}B_1(N,M)\sumstar_{\id{a}\sim M} \tq \lambda^g_\id{a}\tq^2 \\
&= \sum_{g\in G}B_1(N,M)\sumstar_{\id{a}\sim M} \tq \lambda_\id{a}\tq^2 \\
&= \tq G\tq B_1(N,M)\sumstar_{\id{a}\sim M} \tq \lambda_\id{a}\tq^2.
\end{align*}
On the other hand, by the duality principle ($\S 4$ of \cite{Mo}), $B_1(M,N)$ is the minimal positive number satisfying
\begin{equation}\label{eq:lem-duality:2}
\sumstar_\id{b}\Big\vert \sumstar_{\id{a}\sim M} \lambda_\id{a}\chi_\id{b}(\id{a}) \Big\vert^2 \le B_1(M,N) \sumstar_{\id{a}\sim M} \tq \lambda_\id{a}\tq^2
\end{equation}
for every sequence $\lambda_\id{a}$. It follows that $B_1(M,N)\le \tq G\tq B_1(N,M)$.
\end{proof}

\noindent
We can now prove Lemma \ref{lemme:initial-bound} as promised.
\begin{proof}[Proof of Lemma \ref{lemme:initial-bound}]
Our first goal is to remove the $^*$-restriction on the $\id{a}$ sum in $B_1(M,N)$. We accomplish this by introducing weights of the form
\[
\rho_{a,b}(t) :=
\int_{(\sigma)} \Gamma\left(\frac{s}{2}\right)^{a} \Gamma(s)^{b} t^{-s} \,ds, \quad \sigma >0 .
\]
Observe that these weights are positive: Parseval's formula for the Mellin transform gives 
\[
\rho_{a,b}(t)
= 2^{b}\idotsint\limits_{\substack{x_i, y_j= 0\\x_1 \pts x_a y_1 \pts y_b=1}}^{\qquad \qquad \infty } e^{-\left( \frac{t^2}{y_1^2}+\frac{1}{y_2^2}+\pts + \frac{1}{y_b^2} + \frac{1}{x_1} +\pts + \frac{1}{x_a}\right)} y_1 \, dx_1 \pts dx_a dy_1 \pts dy_b > 0 .
\]
Moreover, $\rho_{a,b}(t)$ attains a (positive) minimum on the compact set $[1,2]$, whence
\[
\rho_{a,b}\bigg(\frac{\rond{N}\id{a}}{M}\bigg) \gg_{a,b} 1 \qquad \text{for } \id{a} \sim M .
\]
It follows that
\[
\begin{split}
\sumstar_{\id{a} \sim M} \Big\vert \ \sumstar_{\id{b}} \lambda_\id{b} \chi_\id{b}(\id{a})\Big\vert^2 
&\ll 
\sumstar_{\id{a} \sim M} \rho_{r_1, r_2}\bigg(\frac{\rond{N}\id{a}}{M}\bigg)
\Big\vert \ \sumstar_{\id{b}} \lambda_\id{b} \chi_\id{b}(\id{a})\Big\vert^2 \\
&\ll
\sum_{\id{a} \neq 0} \int_{(\sigma)} \Gamma\left(\frac{s}{2}\right)^{r_1} \Gamma(s)^{r_2} \left(\frac{\rond{N}\id{a}}{M}\right)^{-s} \,ds \ \Big\vert \sumstar_\id{b} \lambda_\id{b} \chi_\id{b}(\id{a})\Big\vert^2 .
\end{split}
\]
(Note that the implicit constant depends on the field $k$, but nothing else.) Expanding the last expression, we are left with sums of the shape
\begin{equation*}
\theta(\chi,M) = \sum_{\id{a}\neq 0} \chi(\id{a}) \int_{(\sigma)} \Gamma\left(\frac{s}{2}\right)^{r_1} \Gamma(s)^{r_2} \left(\frac{\rond{N}\id{a}}{M}\right)^{-s} \,ds,
\end{equation*}
where $\chi=\chi_{\id{b}_1} \overline{\chi}_{\id{b}_2}$ with $\id{b}_1$ and $\id{b}_2$ squarefree.\\

For a non-principal Hecke character $\chi\mod{\id{m}}$, one has $\theta(\chi,M) \ll (\rond{N}\id{m})^{(1+\varepsilon)/2}$; this can be seen by expressing $\theta(\chi,M)$ as the inverse Mellin transform of the completed $L$-function $\Lambda(\chi,s)$ (see \cite[Lemma 2]{HBP} for the case of cubic characters). Note that $\chi_{\id{b}_1}\barre{\chi}_{\id{b}_2}$ is principal only for $\id{b}_1=\id{b}_2$, in which case $\theta(\chi,M)$ is estimated trivially by $M$. Thus, one obtains the upper bound

\begin{equation*}
\sumstar_{\id{a} \sim M} \Big\vert \ \sumstar_{\id{b}} \lambda_\id{b} \chi_\id{b}(\id{a})\Big\vert^2 \ll M \sum_\id{b} \tq \lambda_\id{b}\tq^2  + N^{1+\varepsilon} \Big\vert \ \sumstar_{\id{b}_1,\id{b}_2} \lambda_{\id{b}_1} \barre{\lambda}_{\id{b}_2} \Big\vert.
\end{equation*}
An application of the Cauchy-Schwarz inequality concludes the proof.
\end{proof}

We conclude this section by deducing Theorem \ref{thm2} from Theorem \ref{cor}.

\begin{proof}[Proof of Theorem \ref{thm2}]
Suppose $(E_\alpha)$ holds, and let $C$ be the absolute constant appearing in Lemma \ref{lemme:increasing}. We consider two cases. If $N^{2-1/\alpha}<C M\log(2MN)$, then $N^{2\alpha-1}M^{1-\alpha}\ll M^{1+\varepsilon}$, and Theorem \ref{cor} implies that $B_1(M,N)\ll (MN)^\varepsilon (M+N)$. If $N^{2-1/\alpha}\ge C M\log(2MN)$, then Lemma \ref{lemme:increasing} implies $B_1(M,N) \ll B_1(N^{2-1/\alpha},N)$, which by Theorem \ref{cor} is bounded by $(MN)^\varepsilon (N^{2-1/\alpha})$. In either case, we conclude $B_1(M,N)\ll (MN)^\varepsilon (M+ N^{2-1/\alpha})$.
\end{proof}

\begin{rque}
A formal application of the Poisson summation formula (ignoring restrictions to squarefree entries, etc.) gives
\begin{equation}\label{eq:conj}
B_1(M,N) \ll \frac{M}{N}  B_1\left(\frac{N^2}{M},N\right)
\end{equation}
independently of $\chi_\id{a}$ being a quadratic character. Applying first Lemma \ref{lemme:duality} and then $(E_\alpha)$ to the right hand side of \eqref{eq:conj} yields
$B_1(M,N) \ll M+ N^{2\alpha-1}M^{1-\alpha}$. This is precisely the bound given in Theorem \ref{cor}. The reason our argument is significantly more complicated is the presence of the squarefree restriction on the sums, which prevents us from directly applying Poisson summation.
In the following section, we introduce the machinery we use to get around this obstruction. Note that if the squarefree condition was removed in the definition \eqref{eq:defB1} of $B_1$, then the main theorem, i.e., the bound $B_1(M,N) \ll (MN)^\varepsilon (M+N)$, would not hold any longer; indeed, as it has already been noticed in \cite[p. 236]{HB}, considering the sequence $(\lambda_\id{b})$ defined by $\lambda_\id{b}=1$ if $\id{b}$ is a square ideal and $\lambda_\id{b}=0$ otherwise, one sees that the quantity
\[
\sum_{\id{a} \sim M} \Big\vert \ \sum_{\id{b}} \lambda_\id{b} \chi_\id{b}(\id{a})\Big\vert^2
\]
is of order $MN$.
\end{rque}

\begin{rque}
Lemmas \ref{lemme:initial-bound}, \ref{lemme:increasing}, and \ref{lemme:duality} do not depend on the Hecke family being quadratic, and hold therefore for any $n$-th order Hecke family.
\end{rque}

\section{Two related sums and a reduction of Theorem \ref{cor}} \label{sec:Reduction}

To prove Theorem \ref{cor}, we follow Heath-Brown and consider two companion sums to $B_1(M,N)$. Given a fixed ideal $\id{c}$ and a sequence $\lambda$ of support $N$, define

\begin{equation}\label{eq:defB2}
\Sigma_2(M,N,K,\lambda) = 
\sum_{\substack{\id{a}\sim M\\\id{a}\in I(\id{c})\\s(\id{a})>K}} \Big\vert \ \sumstar_\id{b} \lambda_\id{b} \chi_\id{b}(\id{a})\Big\vert^2,
\end{equation}
where $s(\id{a})$ denotes the norm of the squarefree part of $\id{a}$ coprime to $\id{c}$. (In other words, if $\id{a}=\id{a}_1\id{a}_2\id{a}_3^2$ with $\id{a}_1$, $\id{a}_2$ squarefree, $\id{a}_1$ divides $\id{c}$ and $(\id{a}_2,\id{c})=(1)$, then $s(\id{a})=\rond{N} \id{a}_2$.) Let $B_2(M,N,K)$ be the supremum taken over all sequences of support $N$ with ${\norme\lambda\norme=1}$.
Note that $B_1(M,N)= B_2(M,N,M)\le B_2(M,N,K)$, for any $0<K\le M$. Next, for an ideal $\id{g}\in I(\id{c})$ and a class $g\in G$, let

\begin{equation}\label{eq:defB3}
\Sigma_3(M,N,K,\id{g},g,\lambda) = 
\sumstar_{\substack{(\id{b}_1,\id{b}_2)=\id{g}\\ [\id{b}_1]=[\id{b}_2]=g}} \lambda_{\id{b}_1} \barre{\lambda}_{\id{b}_2} \sum_{\substack{\id{a}\in I(\id{c})\\s(\id{a})>K}} W\left(\frac{\rond{N}(\id{a})}{M}\right) \chi_{\id{b}_1}(\id{a}) \barre{\chi}_{\id{b}_2}(\id{a}),
\end{equation}
where $W : \R_{\geq 0} \to \R_{\geq 0}$ is a smooth weight function with support $[1/2, 5/2]$. Set $B_3(M,N,K,\id{g})$ to be the supremum of $\Sigma_3$ taken over classes $g$ and over all sequences of support $N$ with $\norme\lambda\norme=1$.\\

The \textit{raison d'\^etre} of $B_3$ is that we can apply Poisson summation to it, which we cannot do directly to $B_1$. It is clear that $B_1$, $B_2$, and $B_3$ are closely related quantities, and we will show how to pass back and forth between them. We start by giving a relation between $B_2$ and $B_3$:

\begin{lemme}\label{lemme:B2-B3}
Let $\varepsilon>0$. Let $M,N\ge 1$. Then for any $K\le M/2$ and any integral ideal $\id{g}_0$ with $1\le \rond{N}\id{g}_0\le N$, there exists $1\le N_1\le N\rond{N}(\id{g}_0)^{-1}$ such that

\begin{equation*}
B_2(M,N,K)
\ll N^\varepsilon B_2(M,N_1,K) + \sum_{\rond{N}\id{g} \le \rond{N}\id{g}_0}B_3(M,N,K,\id{g}).
\end{equation*}
\end{lemme}

\begin{proof}
Let $\lambda$ be a sequence of support $N$. By inserting positive weights $W(\rond{N}\id{a}/M)$ and using the Cauchy-Schwarz inequality, we have

\begin{align*}
\Sigma_2(M,N,K,\lambda)
&\ll \sum_{\substack{\id{a}\in I(\id{c})\\ s(\id{a}) >K}} W\left(\frac{\rond{N}\id{a}}{M}\right) \Big\vert \sumstar_\id{b} \lambda_\id{b} \chi_\id{b}(\id{a}) \Big\vert^2\\
&\ll \max_{g\in G} \sum_{\substack{\id{a}\in I(\id{c})\\ s(\id{a}) >K}} W\left(\frac{\rond{N}\id{a}}{M}\right) \Big\vert \sumstar_{[\id{b}]=g} \lambda_\id{b} \chi_\id{b}(\id{a}) \Big\vert^2.
\end{align*}
By opening the square and sorting the terms according to their greatest common divisor, we obtain

\begin{align*}
\Sigma_2(M,N,K,\lambda)
&\ll \sum_{\id{g}} \max_{g\in G} \sum_{\substack{\id{a}\in I(\id{c})\\ s(\id{a}) >K}} W\left(\frac{\rond{N}\id{a}}{M}\right) \sumstar_{\substack{(\id{b}_1,\id{b}_2)=\id{g}\\ [\id{b}_1]=[\id{b}_2]=g}} \lambda_{\id{b}_1} \barre{\lambda}_{\id{b}_2} \chi_{\id{b}_1}\barre{\chi}_{\id{b}_2} (\id{a})\\
&\ll \sum_{\rond{N}\id{g}\le \rond{N}\id{g}_0} B_3(M,N,K,\id{g}) \norme \lambda\norme \\
& +  \sum_{\rond{N}\id{g}> \rond{N}\id{g}_0} \max_{g\in G} \sum_{\substack{\id{a}\in I(\id{c})\\ s(\id{a}) >K}} W\left(\frac{\rond{N}\id{a}}{M}\right) \sumstar_{\substack{(\id{b}_1,\id{b}_2)=\id{g}\\ [\id{b}_1]=[\id{b}_2]=g}} \lambda_{\id{b}_1} \barre{\lambda}_{\id{b}_2} \chi_{\id{b}_1}\barre{\chi}_{\id{b}_2} (\id{a}).
\end{align*}
Using the M\"obius function to detect coprimality, we see that for each ideal $\id{g}$,

\begin{align*}
&\Big\vert\max_{g\in G} \sum_{\substack{\id{a}\in I(\id{c})\\ s(\id{a}) >K}} W\left(\frac{\rond{N}\id{a}}{M}\right) \sumstar_{\substack{(\id{b}_1,\id{b}_2)=\id{g}\\ [\id{b}_1]=[\id{b}_2]=g}} \lambda_{\id{b}_1} \barre{\lambda}_{\id{b}_2} \chi_{\id{b}_1}\barre{\chi}_{\id{b}_2} (\id{a})\Big\vert\\
& \le \sum_\id{d} \max_{g\in G} \sum_{\substack{\id{a}\in I(\id{c})\\ s(\id{a}) >K}} W\left(\frac{\rond{N}\id{a}}{M}\right) \Big\vert \ \sumstar_{\substack{\id{b}_1,\id{b}_2\congru 0 \mod{\id{d}} \\ [\id{b}_1]=[\id{b}_2]=g[\id{g}^{-1}]}} \lambda_{\id{g}\id{b}_1} \barre{\lambda}_{\id{g}\id{b}_2} \chi_{\id{b}_1}\barre{\chi}_{\id{b}_2} (\id{a})\Big\vert.
\end{align*}
Consider $\lambda^{(\id{g},g)}_\id{b}$, defined to be $\lambda_\id{gb}$ if $[\id{b}]=g[\id{g}^{-1}]$ and $0$ otherwise; this is a sequence of support $N/\rond{N}\id{g}$. The preceding display is then equal to

\begin{align*}
&\sum_\id{d} \max_{g\in G} \sum_{\substack{\id{a}\in I(\id{c})\\ s(\id{a}) >K}} W\left(\frac{\rond{N}\id{a}}{M}\right) \Big\vert \ \sumstar_{\id{b}\congru 0 \mod{\id{d}}} \lambda^{(\id{g},g)}_\id{b} \chi_\id{b}(\id{a})\Big\vert^2\\
&\le B_2\left(M,\frac{N}{\rond{N}\id{g}},K\right) \sum_\id{d} \max_{g\in G}   \sumstar_{\substack{\id{b}\congru 0 \mod{\id{d}}\\ [\id{gb}]=g}} \tq \lambda_\id{gb}\tq^2\\
&\le B_2\left(M,\frac{N}{\rond{N}\id{g}},K\right) \max_{g\in G}\sumstar_{[\id{b}]=g} \tq \lambda_\id{b} \tq^2 \tau(\id{b})^2,
\end{align*}
where $\tau$ is the divisor function. The conclusion easily follows.
\end{proof}

The principal difficulty is to relate $B_3$ to $B_1$. To do this, we will prove (in the following two sections) the following. Recall the bound
\begin{equation*}\tag{$E_\alpha$}
B_1(M,N)\ll (MN)^\varepsilon (M+N^\alpha), \qquad \ec{for all } M,N \geq 1 \ec{ and } \varepsilon>0.
\end{equation*}

\begin{prop}\label{prop}
Assume $(E_\alpha)$. Let $M,N \ge 1$ and $\varepsilon>0$. Then
\begin{equation*}
B_3(M,N,K,\id{g}) \ll \rond{N}(\id{g})^4 (MN)^\varepsilon \left(M + N +\sqrt{M}K^{\alpha-1/2} +\sqrt{\frac{M}{K}} N\right)
\end{equation*}
whenever 
$N^2 M^{-1} (MN)^\varepsilon \le K\le M (MN)^{-\varepsilon}$.
\end{prop}

\noindent
We conclude this section by showing that the proposition implies Theorem \ref{cor}.

\begin{proof}[Proof of Theorem \ref{cor}] 
Let $M$, $N\ge 1$ and $\varepsilon$ be fixed. Assume first that $N(MN)^\varepsilon \le M$ and define $K=N^2M^{-1} (MN)^\varepsilon$. Let $r\ge \varepsilon^{-1}$ be an integer and define $g=N^{1/r}$. Lemma \ref{lemme:B2-B3} and Proposition \ref{prop} allow us to define a sequence $N_i$ as follows: $N_0=N$, $N_{i+1} \le N_i g^{-1}$ and

\begin{equation}
B_2(M,N_i,K) \ll (MN)^\varepsilon \Big(B_2(M,N_{i+1},K)
	+ g^5(M+N_i+N_i^{2\alpha-1}M^{1-\alpha})\Big).
\end{equation}
After iterating, we obtain

\begin{align}
B_2(M,N,K)
&\ll (MN)^{i\varepsilon} \left(B_2(M,N_i,K) + \sum_{j=0}^{i-1} g^5(M+N_j+N_j^{2\alpha-1}M^{1-\alpha})\right)\nonumber\\
&\ll (MN)^{i\varepsilon} \left(B_2(M,N_i,K) + ig^5(M+N+N^{2\alpha-1}M^{1-\alpha})\right).\label{concl:eq:1}
\end{align}
Note that $N_r\le1\le g$, thus \eqref{concl:eq:1} with $i=r$ combined with the estimate $B_2(M,N,K)\le MN$ gives

\[
B_2(M,N,K) \ll r (MN)^{r\varepsilon} g^5(M+N+N^{2\alpha-1}M^{1-\alpha}). \qedhere
\]
\end{proof}

The remainder of the paper is devoted to the proof of Proposition \ref{prop}: in Section \ref{sec:poisson} we determine an explicit formula for $B_3$, and in Section \ref{sec:proof} we study this formula term by term to deduce Proposition \ref{prop}.

\section{An explicit formula for the norm $B_3$}\label{sec:poisson}

The aim of this section is to prove formula \eqref{eq:explicit}, by applying Poisson summation formula. In doing so, the fact that we work with quadratic characters turns out to be crucial. We shall make frequent use of certain transforms $\hat{h}$, $\dot{h}$ and $\ddot{h}$, all described in Section \ref{sec:app}.\\

We define the following quantity:
\begin{equation}\label{eq:defB4}
\Sigma_4(\id{m},?K;h,X,\chi) = \sum_{\substack{\id{a}\in I(\id{m})\\s(\id{a})?K}} h\left(\frac{\rond{N}(\id{a})}{X}\right) \chi(\id{a}),
\end{equation}
where $\chi$ is a Hecke character, $\id{m}$ is an ideal, $h$ is any function, and $?$ stands in for $\le$ or $>$. Lemma \ref{lemme:Poisson-char} immediately gives:
\begin{cor}\label{cor:Poisson-char}
For any primitive quadratic character $\chi \mod{\id{f}}$, 
\[
\Sigma_4\big((1),\ge1;h,X,\chi\big) = \frac{X}{\sqrt{\rond{N} \id{f}}} \Sigma_4\left((1),\ge1;\dot{h},\frac{\rond{N}\id{f}}{X}, \chi\right).
\]
\end{cor}

\subsection{Decomposition of $B_3$}\label{subsec:B3-B4}

Let $\lambda$ be a sequence of support $N$, and suppose $\id{g}\in I(\id{c})$ and $g\in G$.
Recall the definitions \eqref{eq:defB3} and \eqref{eq:defB4}, and let $\id{s}$ be the radical of $\id{c}$ (i.e.\ the product of all prime ideals dividing $\id{c}$). We have

\[
\begin{split}
\Sigma_3(M,N,K,\id{g},g,\lambda)
&= \sumstar_{\substack{(\id{b}_1,\id{b}_2)=\id{g} \\ [\id{b}_1]=[\id{b}_2]=g}} \lambda_{\id{b}_1} \barre{\lambda}_{\id{b}_2} \sum_{\substack{\id{a}\in I(\id{s})\\ s(\id{a})> K}} W\left(\frac{\rond{N}(\id{a})}{M}\right) \chi_{\id{b}_1}(\id{a}) \chi_{\id{b}_2}(\id{a})\\
&= \sumstar_{\substack{(\id{b}_1,\id{b}_2)=(1)\\  [\id{b}_1]=[\id{b}_2]=g[\id{g}^{-1}]}} \lambda_{\id{gb}_1} \barre{\lambda}_{\id{gb}_2}
\Sigma_4\left(\id{s g}, >K;W,M,\chi_{\id{b}_1}\chi_{\id{b}_2}\right)\\
&= \sumstar_{\substack{(\id{b}_1,\id{b}_2)=(1)\\  [\id{b}_1]=[\id{b}_2]=g[\id{g}^{-1}]}} \lambda_{\id{gb}_1} \barre{\lambda}_{\id{gb}_2}
\Sigma_4\left(\id{s g}, \ge1;W,M,\chi_{\id{b}_1}\chi_{\id{b}_2}\right) \\
& \phantom{HHHHHHHHHHHH}
- \sumstar_{\substack{(\id{b}_1,\id{b}_2)=(1)\\  [\id{b}_1]=[\id{b}_2]=g[\id{g}^{-1}]}} \lambda_{\id{gb}_1} \barre{\lambda}_{\id{gb}_2}
\Sigma_4\left(\id{s g}, \le K;W,M,\chi_{\id{b}_1}\chi_{\id{b}_2}\right) .
\end{split}
\]
Note that by Property (3) of Definition \ref{defi:HeckeFamily}, since
$[\id{b}_1]=[\id{b}_2]$ we know that the character $\chi_{\id{b}_1}\chi_{\id{b}_2}$ is primitive modulo $\id{b}_1\id{b}_2$. 
Once we remove the coprimality condition (using the M\"obius function), we are in a position to apply the Poisson summation formula (Corollary \ref{cor:Poisson-char}) to $\Sigma_4(\id{s g}, \ge 1,W,M,\chi_{\id{b}_1,\id{b}_2})$. We therefore have
\begin{equation*}
\Sigma_4(\id{s g}, \ge 1,W,M,\chi_{\id{b}_1}\chi_{\id{b}_2})
= \sum_{\id{e}\mid \id{s g}} \frac{\mu(\id{e})}{\rond{N}\id{e}} \chi_{\id{b}_1}\chi_{\id{b}_2}(\id{e}) \frac{M}{\sqrt{\rond{N}\id{b}_1\id{b}_2}} \Sigma_4\left((1),\ge 1;\dot{W},\frac{\rond{N}(\id{e}\id{b}_1\id{b}_2)}{M},\chi_{\id{b}_1}\chi_{\id{b}_2}\right),
\end{equation*}
whence
\begin{align*}
&\Sigma_3(M,N,K,\id{g},g,\lambda) = \\
& \phantom{HHH}
\sumstar_{\substack{(\id{b}_1,\id{b}_2)=(1)\\  [\id{b}_1]=[\id{b}_2]=g[\id{g}^{-1}]}} \lambda_{\id{gb}_1} \barre{\lambda}_{\id{gb}_2}\sum_{\id{e}\mid \id{s g}} \frac{\mu(\id{e})}{\rond{N}\id{e}} \chi_{\id{b}_1}\chi_{\id{b}_2}(\id{e}) \frac{M}{\sqrt{\rond{N}\id{b}_1\id{b}_2}} \Sigma_4\left((1),\ge 1;\dot{W},\frac{\rond{N}(\id{e}\id{b}_1\id{b}_2)}{M},\chi_{\id{b}_1}\chi_{\id{b}_2}\right)\\
& \phantom{HHHHH}
- \sumstar_{\substack{(\id{b}_1,\id{b}_2)=(1)\\  [\id{b}_1]=[\id{b}_2]=g[\id{g}^{-1}]}} \lambda_{\id{gb}_1} \barre{\lambda}_{\id{gb}_2} \Sigma_4(\id{s g}, \le K;W,M,\chi_{\id{b}_1}\chi_{\id{b}_2}).
\end{align*}
Using $\dot{W}(x)\ll \tq x\tq^{-A}$ for $\tq x\tq>1$, one shows that

\begin{equation}\label{eq:B3-B4}
\begin{split}
\Sigma_3(M,N,K,\id{g},g,\lambda)
&= \sumstar_{\substack{(\id{b}_1,\id{b}_2)=(1)\\  [\id{b}_1]=[\id{b}_2]=g[\id{g}^{-1}]}} \lambda_{\id{gb}_1} \barre{\lambda}_{\id{gb}_2}\sum_{\id{e}\mid \id{sg}} \frac{\mu(\id{e})}{\rond{N}\id{e}} \chi_{\id{b}_1}\chi_{\id{b}_2}(\id{e}) \frac{M}{\sqrt{\rond{N}\id{b}_1\id{b}_2}} \\
&\hspace{3cm}
\times \Sigma_4\left((1),\le K;\dot{W},\frac{\rond{N}(\id{e}\id{b}_1\id{b}_2)}{M},\chi_{\id{b}_1}\chi_{\id{b}_2}\right)\\
&- \sumstar_{\substack{(\id{b}_1,\id{b}_2)=(1)\\  [\id{b}_1]=[\id{b}_2]=g[\id{g}^{-1}]}} \lambda_{\id{gb}_1} \barre{\lambda}_{\id{gb}_2} \Sigma_4(\id{sg}, \le K;W,M,\chi_{\id{b}_1}\chi_{\id{b}_2}) + \rond{O}_\varepsilon\left(\norme \lambda\norme \right).
\end{split}
\end{equation}
Formula \eqref{eq:B3-B4} is an inexplicit version of Proposition \ref{prop}.
We now turn to the quantity ${\Sigma_4(\id{m},\le K;h,X,\chi)}$, keeping in mind our choice of parameters

\begin{equation*}
\left\{
\begin{aligned}
&\id{m}=(1), &h=\dot{W}, &\ X=\frac{\rond{N}(\id{e} \id{b}_1\id{b}_2)}{M}, &\chi= \chi_{\id{b}_1}\chi_{\id{b}_2} &\qquad (1)\\
&\id{m}=\id{sg}, &h=W, &\ X=M, &\chi=\chi_{\id{b}_1}\chi_{\id{b}_2} &\qquad (2)
\end{aligned}
\right.
\end{equation*}

\subsection{An explicit formula for $\Sigma_4$}\label{subsec:explicitformula}

Let $\chi$ be a quadratic primitive Hecke character of conductor $\id{f}$. Let $\id{m}$ be an integral ideal of $K$. Let $h$ be a smooth function with compact support and write $f(x):= h(x^2)$. Then

\begin{align*}
\Sigma_4(\id{m}, \le K; h,X,\chi)
&= \sum_{\substack{(\id{a},\id{m})=(1) \\ s(\id{a})\le K}} h\left(\frac{\rond{N}\id{a}}{X}\right) \chi(\id{a})\\
&= \sum_{\substack{\id{a} \ec{ sq-free}\\ (\id{a},\id{m})=(1)\\ s(\id{a})\le K}} \chi(\id{a}) \sum_{\substack{(\id{b},\id{m})=(1)\\ (\id{b},\id{f})=(1)}} h\left(\frac{\rond{N}\id{b}^2}{X/\rond{N} \id{a}}\right)\\
&= \sum_{\substack{\id{a} \ec{ sq-free}\\ (\id{a},\id{m})=(1)\\ s(\id{a})\le K}} \chi(\id{a}) \sum_{(\id{b},\id{m}\id{f})=(1)} f\left(\frac{\rond{N}\id{b}}{\sqrt{X/\rond{N} \id{a}}}\right).
\end{align*}
In order to apply Lemma \ref{lemme:Poisson-coprim}, we make a dyadic partition for the $\id{a}$-sum into intervals $(B,2B]$, with $B\le K$. Then, for any $B\le K$, $Y_B\le \sqrt{X/2B}$, $Z_B\ge \sqrt{X/B}$ and $L_B\ge 2Z_B^2B/X$, we have (recall the notations defined in Section \ref{sec:app}) 

\begin{equation}\label{eq:B4}
\begin{split}
\Sigma_4(\id{m}, \le K; h,X,\chi)
&=  \sum_{B\le K} \Bigg\{
\frac{\alpha_k}{A_k} \frac{\varphi(\id{mf})}{\rond{N}(\id{mf})}  \sqrt{X} \sum_{\substack{\id{a} \ec{ sq-free}\\ (\id{a},\id{m})=(1)\\ s(\id{a})\sim B}} \frac{\chi(\id{a})}{\sqrt{\rond{N}\id{a}}} \int_\R h(t^2) \,dt\\
&-\frac{\alpha_k}{A_k} \sqrt{X}\sum_{\substack{\id{a} \ec{ sq-free}\\ (\id{a},\id{m})=(1)\\ s(\id{a})\sim B}} \frac{\chi(\id{a})}{\sqrt{\rond{N} \id{a}}}\sum_{\substack{\id{d}\mid \id{mf}\\\rond{N}\id{d} >Z_B}} \frac{\mu(\id{d})}{\rond{N}\id{d}}\int_\R h(t^2) \,dt\\
&- \delta_{d=2} h(0) \alpha_kA_k \sum_{\substack{\id{a} \ec{ sq-free}\\ (\id{a},\id{m})=(1)\\ s(\id{a})\sim B}}\chi(\id{a}) \sum_{\substack{\id{d}\mid \id{mf}\\\rond{N}\id{d}\le Z_B}} \mu(\id{d})\\
&+ \sqrt{X}  \sum_{\substack{\id{a} \ec{ sq-free}\\ (\id{a},\id{m})=(1)\\ s(\id{a})\sim B}} \frac{\chi(\id{a})}{\sqrt{\rond{N}\id{a}}} \sum_{\substack{\id{d} \mid \id{mf}\\ Y_B<\rond{N} \id{d} \le Z_B}} \frac{\mu(\id{d})}{\rond{N} \id{d}} \sum_{\substack{\id{b}\neq 0\\ \rond{N}\id{b}\le L_B}} \ddot{h}\left(\frac{\rond{N}\id{b}\sqrt{X}}{\rond{N}\id{d}\sqrt{\rond{N}\id{a}}}\right)\\
& + \rond{O}\left(\sqrt{\frac{X}{B}}\left(\frac{\sqrt{X}}{\sqrt{B}Y_B}\right)^{-A}\right) +  \rond{O}\left(\sqrt{\frac{X}{B}}\left(\frac{\sqrt{B}Z_B}{\sqrt{X}}\right)^{-A}\right)\Bigg\},
\end{split}
\end{equation}
for any $A>0$, where $\delta$ is Kronecker's delta and $d=[k:\Q]$.

\subsection{An explicit formula for $\Sigma_3$}\label{subsec:explicitformula-bis}

We now obtain an explicit formula for $\Sigma_3(M,N,K,\id{g},g,\lambda)$, by plugging \eqref{eq:B4} into \eqref{eq:B3-B4}. Recall that the conductor of $\chi_{\id{b}_1}\chi_{\id{b}_2}$ is precisely $\id{b}_1 \id{b}_2$ in this case. According to the cases $(1)$ and $(2)$ described above, our choices for $Y_B$, $Z_B$ and $L_B$ are
\begin{equation*}
\left\{\begin{aligned}
& Y_{B,\id{e}}^{(1)}=\frac{N\sqrt{\rond{N}\id{e}}}{\rond{N}\id{g}\sqrt{2BM}}(MN)^{-\varepsilon_1}
&
& Z_{B,\id{e}}^{(1)}= \frac{N\sqrt{\rond{N}\id{e}}}{\rond{N}\id{g}\sqrt{BM}}(MN)^{\varepsilon_1}
&
& L^{(1)}=2 (MN)^{\varepsilon_1} &\qquad (1)
\\
& Y_B^{(2)}=\frac{\sqrt{M}}{\sqrt{2B}} (MN)^{-\varepsilon_1}
&
& Z_B^{(2)}= \frac{\sqrt{M}}{\sqrt{B}} (MN)^{\varepsilon_1}
&
& L^{(2)}=2 (MN)^{\varepsilon_1} &\qquad (2)
\end{aligned}\right.
\end{equation*}
for some $\varepsilon_1>0$ which can be conveniently chosen. We thus obtain (recall that $W(0)=0$)
\begin{multline}\label{eq:explicit}
\Sigma_3(M,N,K,\id{g},g,\lambda) = 
\sum_{B\le K} \bigg\{
T(B,\id{g},g,\lambda) -T'(B,\id{g},g,\lambda) - E_1(B,\id{g},g,\lambda) \\
- E_2(B,\id{g},g,\lambda) + E_3(B,\id{g},g,\lambda) + E_4(B,\id{g},g,\lambda) - E_5(B,\id{g},g,\lambda)\bigg\},
\end{multline}
where the terms in \eqref{eq:explicit}, all depending on $M$ and $N$, are given by
\begin{multline*}
T(B,\id{g},g,\lambda)=
\frac{\alpha_k}{A_k} \sum_{\id{e}\mid \id{sg}}\frac{\mu(\id{e})}{\sqrt{\rond{N}\id{e}}} \sum_{\substack{\id{a} \ec{ sq-free}\\ s(\id{a})\sim B}} \sqrt{\frac{M}{\rond{N}\id{a}}} \  \sumstar_{\substack{(\id{b}_1,\id{b}_2)=(1)\\  [\id{b}_1]=[\id{b}_2]=g[\id{g}^{-1}]}} \lambda_{\id{gb}_1} \barre{\lambda}_{\id{gb}_2}\frac{\varphi(\id{b}_1\id{b}_2)}{\rond{N} \id{b}_1\id{b}_2} \chi_{\id{b}_1}\chi_{\id{b}_2}(\id{ae}) \int_\R \dot{W}(x^2)\,dx,
\end{multline*}

\begin{multline*}
T'(B,\id{g},g,\lambda)=
\frac{\alpha_k}{A_k} \ \sumstar_{\substack{(\id{a},\id{g})=(1)\\ \id{a}\sim B}} \sqrt{\frac{M}{\rond{N}\id{a}}} \ \sumstar_{\substack{(\id{b}_1,\id{b}_2)=(1)\\  [\id{b}_1]=[\id{b}_2]=g[\id{g}^{-1}]}} \lambda_{\id{gb}_1} \barre{\lambda}_{\id{gb}_2}\frac{\varphi(\id{sg}\id{b}_1\id{b}_2)}{\rond{N}\id{sg}\id{b}_1\id{b}_2} \chi_{\id{b}_1}\chi_{\id{b}_2}(\id{a}) \int_\R W(x^2)\,dx,
\end{multline*}

\begin{multline*}
E_1(B,\id{g},g,\lambda)=
\frac{\alpha_k}{A_k} \sum_{\id{e}\mid \id{sg}}\frac{\mu(\id{e})}{\sqrt{\rond{N}\id{e}}} \sum_{\substack{\id{a} \ec{ sq-free}\\ s(\id{a})\sim B}} \sqrt{\frac{M}{\rond{N}\id{a}}} \sum_{\rond{N}\id{d} > Z_{B,\id{e}}^{(1)}} \frac{\mu(\id{d})}{\rond{N}\id{d}} \sumstar_{\substack{(\id{b}_1,\id{b}_2)=(1)\\ [\id{b}_1]=[\id{b}_2]=g[\id{g}^{-1}]\\ \id{b}_1\id{b}_2 \congru 0\mod{\id{d}}}} \lambda_{\id{gb}_1} \barre{\lambda}_{\id{gb}_2}\chi_{\id{b}_1}\chi_{\id{b}_2}(\id{ae})\\
\times \int_\R \dot{W}(x^2)\,dx,
\end{multline*}

\begin{equation*}
E_2(B,\id{g},g,\lambda)=\\
\delta_{d=2}\alpha_kA_k\dot{W}(0) M \sum_{\id{e}\mid \id{sg}}\frac{\mu(\id{e})}{\rond{N}\id{e}} \sum_{\substack{\id{a} \ec{ sq-free}\\ s(\id{a})\sim B}} \sum_{\rond{N}\id{d} \le Z_{B,\id{e}}^{(1)}} \mu(\id{d}) \sumstar_{\substack{(\id{b}_1,\id{b}_2)=(1)\\ [\id{b}_1]=[\id{b}_2]=g[\id{g}^{-1}]\\\id{b}_1\id{b}_2\congru 0\mod{\id{d}}}} \frac{\lambda_{\id{gb}_1} \barre{\lambda}_{\id{gb}_2}}{\sqrt{\rond{N} \id{b}_1\id{b}_2}}\chi_{\id{b}_1}\chi_{\id{b}_2}(\id{ae}),
\end{equation*}

\begin{multline*}
E_3(B,\id{g},g,\lambda)= \sum_{\id{e}\mid \id{sg}}\frac{\mu(\id{e})}{\sqrt{\rond{N}\id{e}}} \sum_{\substack{\id{a} \ec{ sq-free}\\ s(\id{a})\sim B}} \sqrt{\frac{M}{\rond{N}\id{a}}} \sum_{Y_{B,\id{e}}^{(1)}<\rond{N}\id{d}\le Z_{B,\id{e}}^{(1)}}
\frac{\mu(\id{d})}{\rond{N}\id{d}} \!\!\!\!\!\!\sumstar_{\substack{(\id{b}_1,\id{b}_2)=(1)\\ [\id{b}_1]=[\id{b}_2]=g[\id{g}^{-1}] \\ \id{b}_1\id{b}_2\congru 0\mod{\id{d}}}} \hspace{-0.5cm} \lambda_{\id{gb}_1} \barre{\lambda}_{\id{gb}_2} \chi_{\id{b}_1}\chi_{\id{b}_2}(\id{ae}) \\
\times \sum_{\substack{\id{b}\neq 0\\ \rond{N}\id{b} \le L^{(1)}}} \ddot{\dot{W}}\left(\frac{\rond{N}\id{b}\sqrt{\rond{N}\id{e}\id{b}_1\id{b}_2}}{\rond{N}\id{d}\sqrt{M\rond{N}\id{a}}}\right),
\end{multline*}

\begin{multline*}
E_4(B,\id{g},g,\lambda)=
\frac{\alpha_k}{A_k} \sumstar_{\substack{(\id{a},\id{g})=(1)\\ \id{a} \sim B}} \sqrt{\frac{M}{\rond{N} \id{a}}} \sum_{\rond{N}\id{d}>Z_B^{(2)}} \frac{\mu(\id{d})}{\rond{N}\id{d}} \sumstar_{\substack{(\id{b}_1,\id{b}_2)=(1)\\ [\id{b}_1]=[\id{b}_2]=g[\id{g}^{-1}]\\\id{sg}\id{b}_1\id{b}_2\congru 0\mod{\id{d}}}} \lambda_{\id{gb}_1} \barre{\lambda}_{\id{gb}_2} \chi_{\id{b}_1}\chi_{\id{b}_2}(\id{a}) \int_\R W(x^2)\,dx
\end{multline*}
and

\begin{multline*}
E_5(B,\id{g},g,\lambda)=\sumstar_{\substack{(\id{a},\id{g})=(1)\\ \id{a} \sim B}} \sqrt{\frac{M}{\rond{N}\id{a}}}\sum_{Y_B^{(2)}<\rond{N} \id{d} \le Z_B^{(2)}}
\frac{\mu(\id{d})}{\rond{N} \id{d}} \sumstar_{\substack{(\id{b}_1,\id{b}_2)=(1)\\ [\id{b}_1]=[\id{b}_2]=g[\id{g}^{-1}]\\ \id{sg}\id{b}_1\id{b}_2\congru 0\mod{\id{d}}}} \lambda_{\id{gb}_1} \barre{\lambda}_{\id{gb}_2}  \chi_{\id{b}_1}\chi_{\id{b}_2}(\id{a}) \\
\times \sum_{\substack{\id{b}\neq 0\\ \rond{N}\id{b} \le L^{(2)}}} \ddot{W}\left(\frac{\rond{N}\id{b}\sqrt{M}}{\rond{N}\id{d}\sqrt{\rond{N}\id{a}}}\right).
\end{multline*}

\section{Proof of Proposition \ref{prop}}\label{sec:proof}

In this section, we study in detail the terms appearing in \eqref{eq:explicit}; this is achieved in Lemma \ref{lemme:error} and Lemma \ref{lemme:main}. The proof of Proposition \ref{prop} will then easily follow.

\subsection{The error terms}\label{subsec:erro}

The following useful result is adapted from \cite[Lemma 10]{HB}.

\begin{lemme}\label{lemme:separation}
Let $M,N,D>0$. Let $\lambda$ and $\lambda'$ be two sequences of support $N$.  Let

\begin{equation*}
S = \sum_{\id{d} \sim D} \sumstar_{\id{a} \sim M} \Big\vert \sumstar_{\substack{(\id{b}_1,\id{b}_2)=(1)\\\id{b}_1\id{b}_2 \congru 0 \mod{\id{d}}}} \lambda_{\id{b}_1} \lambda_{\id{b}_2}' \chi_{\id{b}_1}\chi_{\id{b}_2}(\id{a})\Big\vert .
\end{equation*}
Then there exist $D_1$ and $D_2$ satisfying

\begin{equation*}
\frac{1}{\log (2MN)} \ll D_i\ll D \quad \ec{and}\quad \frac{D}{\log^2(2MN)} \ll D_1D_2 \ll  \frac{D}{\log^2(2MN)},
\end{equation*}
such that

\begin{equation*}
S^2 \ll (MN)^\varepsilon D_1D_2 B_1\left(M,\frac{N}{D_1}\right) B_1\left(M,\frac{N}{D_2}\right)  \norme \lambda\norme \norme \lambda'\norme,
\end{equation*}
for any $\varepsilon >0$.
\end{lemme}

\begin{lemme}\label{lemme:error}
Let $B\le K$. Let $\varepsilon>0$.\\

$(i)$There exist $D_1,D_2\gg (MN)^{-\varepsilon}$ satisfying $D_1D_2 \gg N/ \rond{N}\id{g}\sqrt{BM}$ such that

\begin{equation*}
E_1(B,g,\id{g},\lambda)
\ll_{k,\id{c},W} (BN\rond{N}\id{g})^\varepsilon \sqrt{\frac{M}{B D_1D_2}} B_1\left(B,\frac{N}{D_1\rond{N}\id{g}}\right)^{1/2} B_1\left(B,\frac{N}{D_2\rond{N}\id{g}}\right)^{1/2} \norme \lambda\norme^2.
\end{equation*}

$(ii)$There exist $D_1,D_2\gg (MN)^{-\varepsilon}$ satisfying $D_1D_2 \ll (MN)^\varepsilon N/ \sqrt{MB}$ such that

\begin{align*}
\left.\begin{aligned}
& E_2(B,g,\id{g},\lambda)\\
& E_3(B,g,\id{g},\lambda)
\end{aligned}\right\}
&\ll_{k,\id{c},W} (BN\rond{N}\id{g})^\varepsilon \frac{M}{N} \sqrt{D_1D_2} B_1\left(B,\frac{N}{D_1\rond{N}\id{g}}\right)^{1/2} B_1\left(B,\frac{N}{D_2\rond{N}\id{g}}\right)^{1/2} \norme \lambda\norme^2 .
\end{align*}

$(iii)$There exist $D_1,D_2\gg (MN)^{-\varepsilon}$ satisfying $D_1D_2 \gg\sqrt{M/B}$ such that

\begin{equation*}
\left.\begin{aligned}
& E_4(B,g,\id{g},\lambda)\\
& E_5(B,g,\id{g},\lambda)
\end{aligned}\right\}
\ll_{k,\id{c},W} (BN\rond{N}\id{g})^\varepsilon \sqrt{\frac{M}{B D_1D_2}} B_1\left(B,\frac{N}{D_1\rond{N}\id{g}}\right)^{1/2} B_1\left(B,\frac{N}{D_2\rond{N}\id{g}}\right)^{1/2} \norme \lambda\norme^2.
\end{equation*}
\end{lemme}

\begin{proof}
Since we proceed in the same way for each error term, we only give the details for $E_2(B,g,\id{g},\lambda)$ and $E_4(B,g,\id{g},\lambda)$. Let us start with $E_2(B,g,\id{g},\lambda)$. Up to a constant depending on $k$ and $W$, we have

\begin{equation*}
E_2\ll M \sum_{\substack{\id{a}\ec{ sq-free}\\ s(\id{a})\sim B}} \sum_{\id{e}\mid \id{sg}} \sum_{\rond{N} \id{d} \le Z_{B,\id{e}}^{(1)}} \Big\vert \sumstar_{\substack{(\id{b}_1,\id{b}_2)=(1)\\ [\id{b}_1]=[\id{b}_2]=g[\id{g}^{-1}]\\ \id{b}_1\id{b}_2\congru 0\mod{\id{d}}}} \frac{\lambda_{\id{gb}_1} \barre{\lambda}_{\id{gb}_2}}{\sqrt{\rond{N} \id{b}_1\id{b}_2}} \chi_{\id{b}_1}\chi_{\id{b}_2}(\id{ae})\Big\vert
\end{equation*}
Decompose $\id{a}=\id{a}_1\id{a}_2$ with $\id{a}_1\mid \id{s}$ and $\id{a}_2$ coprime to $\id{s}$. The number of such $\id{a}_1$ depends only on $\id{c}$. The number of terms in the $\id{e}$-sum is $\rond{O}(\rond{N} \id{sg})^\varepsilon$. Define a new sequence $\lambda^{(g,\id{g},\id{e},\id{a}_1)}$ of support $N/\rond{N}\id{g}$ by

\begin{equation*}
\lambda^{(g,\id{g},\id{e},\id{a}_1)}_\id{b}=\begin{cases}
\lambda_\id{gb} (\rond{N}\id{b})^{-1/2}\chi_\id{b}(\id{ea}_1)&\ec{if } [\id{b}]=g[\id{g}^{-1}]\\
0&\ec{otherwise.}
\end{cases}
\end{equation*}
We make a dyadic partition of the $\id{d}$-sum into intervals $\id{d}\sim D$, with $D\le Z_{B,\id{e}}^{(1)}$. Then, for some $\id{e}$ and some $\id{a}_1$,

\begin{equation*}
E_2\ll_{k,\id{c},W} (\rond{N} \id{g})^\varepsilon M \sum_{D\le Z_{B,\id{e}}^{(1)}} \sumstar_{\id{a}\sim B} \sum_{\id{d}\sim D} \Big\vert \sumstar_{\substack{(\id{b}_1,\id{b}_2)=(1)\\\id{b}_1\id{b}_2 \congru 0\mod{\id{d}}}} \lambda^{(g,\id{g},\id{e},\id{a}_1)}_{\id{b}_1} \barre{\lambda}^{(g,\id{g},\id{e},\id{a}_1)}_{\id{b}_2} \chi_{\id{b}_1}\chi_{\id{b}_2}(\id{a})\Big\vert.
\end{equation*}
The number of possible $D$'s is $\rond{O}(\log N)$. Thus, for some $D$, Lemma \ref{lemme:separation} implies the existence of $D_1$ and $D_2$ such that

\begin{equation*}
\frac{1}{\log (BN/\rond{N}\id{g})} \le D_i \le N^2 \quad \ec{and} \quad \frac{D}{\log (BN/\rond{N}\id{g})}\ll D_1D_2\ll \frac{D}{\log (BN/\rond{N}\id{g})}
\end{equation*}
and

\begin{equation*}
E_2 \ll (N\rond{N}\id{g})^\varepsilon M \sqrt{D_1D_2} B_1\left(B,\frac{N}{D_1\rond{N}\id{g}}\right)^{1/2} B_1\left(B,\frac{N}{D_2\rond{N}\id{g}}\right)^{1/2}  \norme \lambda^{(g,\id{g},\id{e},\id{a}_1)}\norme^2.
\end{equation*}
We conclude by observing that $\norme \lambda^{(g,\id{g},\id{e},\id{a}_1)}\norme
\le \norme \lambda \norme \rond{N}\id{g}/N$.
Let us now consider $E_4(B,g,\id{g},\lambda)$. Up to a constant depending on $k$, $\id{c}$, and $W$, we have

\begin{align*}
E_4
&\ll \sqrt{\frac{M}{B}} \sumstar_{\substack{(\id{a},\id{g})=(1)\\ \id{a}\sim B}} \sum_{\rond{N} \id{d} > Z_B^{(2)}} \frac{1}{\rond{N}\id{d}} \Big\vert \sumstar_{\substack{(\id{b}_1,\id{b}_2)=(1)\\ [\id{b}_1]=[\id{b}_2]=g[\id{g}^{-1}]\\ \id{sgb}_1\id{b}_2\congru 0\mod{\id{d}}}}\lambda_{\id{gb}_1} \barre{\lambda}_{\id{gb}_2} \chi_{\id{b}_1}\chi_{\id{b}_2} (\id{a})\Big\vert\\
&\le \sqrt{\frac{M}{B}} \sumstar_{\id{a}\sim B} \sum_{\id{d}_1 \mid \id{sg}} \sum_{\rond{N} \id{d}_2 > Z_B^{(2)}/\rond{N}\id{sg}} \frac{1}{\rond{N}\id{d}_1\id{d}_2} \Big\vert \sumstar_{\substack{(\id{b}_1,\id{b}_2)=(1)\\ [\id{b}_1]=[\id{b}_2]=g[\id{g}^{-1}]\\ \id{b}_1\id{b}_2\congru 0\mod{\id{d}_2}}}\lambda_{\id{gb}_1} \barre{\lambda}_{\id{gb}_2} \chi_{\id{b}_1}\chi_{\id{b}_2} (\id{a})\Big\vert.
\end{align*}
The sum over $\id{d}_1$ is $\rond{O}(\log \rond{N}\id{g})$. We partition the $\id{d}_2$-sum into dyadic intervals, with $\id{d}_2\sim D$ for all $D$ (powers of $2$) lying in the interval $Z_B^{(2)}/\rond{N}\id{sg} < D \le N^2$. We define a new sequence of support $N/\rond{N}\id{g}$

\begin{equation*}
\lambda^{(g,\id{g})}_\id{b}= \begin{cases}
\lambda_\id{gb}&\ec{if } [\id{b}]=g[\id{g}^{-1}]\\
0&\ec{otherwise.}
\end{cases}
\end{equation*}
Note that there are at most $\rond{O}(\log N)$ possible $D$'s. Thus, we have, for some $D$, that

\begin{equation*}
E_4 \ll (N\rond{N}\id{g})^\varepsilon \sqrt{\frac{M}{B}} \frac{1}{D} \sumstar_{\id{a}\sim B} \sum_{\rond{N} \id{d} \sim D}  \Big\vert \sumstar_{\substack{(\id{b}_1,\id{b}_2)=(1)\\ \id{b}_1\id{b}_2\congru 0\mod{\id{d}_2}}}\lambda^{(g,\id{g})}_{\id{b}_1} \barre{\lambda}^{(g,\id{g})}_{\id{b}_2} \chi_{\id{b}_1}\chi_{\id{b}_2} (\id{a})\Big\vert.
\end{equation*}
We conclude by using Lemma \ref{lemme:separation} and the fact that $\norme \lambda^{(g,\id{g})} \norme \le \norme \lambda \norme$.
\end{proof}

\subsection{The main contribution}\label{subsec:main}

First of all, note that if we deal with the terms $T(B,\id{g},\lambda)$ and $T'(B,\id{g},\lambda)$ separately, the technique used previously for the error term does not allow us to conclude, since we would obtain for each of the quantities $T(B,\id{g},g,\lambda)$ and $T'(B,\id{g},g,\lambda)$ the upper bound

\begin{equation*}
\max_{B\le K} \sqrt{\frac{M}{B}} B_1(B,N).
\end{equation*}
If $\id{g}=(1)$, one has equality between the two main terms, as shall be seen below.

\begin{lemme}\label{lemme:main}
Let $M$, $N$, $K$, $\id{g}$ and $g$ be as above. Assume that $\rond{N}\id{g}\ll N$ and $K\le M$. Then, for any $\varepsilon>0$,

\begin{equation*}
\sum_{B\le K} T(B,\id{g},g,\lambda) - T'(B,\id{g},g,\lambda) \ll_{k,\id{c},W}  (MN)^\varepsilon \sqrt{\frac{M}{K}} B_1\Big(K(\rond{N}\id{g})^2(MN)^\varepsilon, N(MN)^\varepsilon\Big) \norme \lambda\norme^2.
\end{equation*}
\end{lemme}

\begin{proof}
Recall that since $\lambda$ is supported by squarefree ideals, one has $(\id{g},\id{b}_1\id{b}_2)=(1)$ in the definition of $T$ and $T'$. By definition of $s(\id{a})$ (given directly below (\ref{eq:defB2})),
$T(B,\id{g},g,\lambda)$ can be written as

\begin{equation*}
\frac{\alpha_k}{A_k} \sumstar_{\rond{N}\id{a}\sim B} \sqrt{\frac{M}{\rond{N}\id{a}}} \ \sumstar_{\substack{(\id{b}_1,\id{b}_2)=(1)\\ [\id{b}_1]=[\id{b}_2]=g[\id{g}^{-1}]}} \lambda_{\id{gb}_1} \barre{\lambda}_{\id{gb}_2}\frac{\varphi(\id{b}_1\id{b}_2)}{\rond{N}\id{b}_1\id{b}_2} \chi_{\id{b}_1}\chi_{\id{b}_2}(\id{a}) \Sigma_5(\chi_{\id{b}_1}\chi_{\id{b}_2},\id{s},\id{g}) \int_0^\infty \dot{W}(x^2)\,dx,
\end{equation*}
where, for a primitive Hecke character $\chi\mod{\id{f}}$ and two ideals $\id{a}$ and $\id{b}$, we set

\begin{equation*}
\Sigma_5(\chi,\id{a},\id{b}) =\sum_{\id{d} \mid \id{a}} \sum_{\id{e}\mid \id{ab}} \frac{\mu(\id{e})}{\sqrt{\rond{N}(\id{de})}} \chi(\id{de}).
\end{equation*}
One easily checks that for $\id{a}$ squarefree, one has

\begin{equation}\label{poisson:eq:main1}
\Sigma_5(\chi,\id{a},(1))=\frac{\varphi(\id{a}_0)}{\rond{N}\id{a}_0},
\end{equation}
where $\id{a}=\id{a}_0\id{a}_\id{f}$, with $(\id{a}_0,\id{f})=(1)$ and $\id{a}_\id{f}\mid \id{f}$. With these notations, it follows that, if $(\id{a},\id{b})=(1)$,

\begin{equation*}
\Sigma_5(\chi,\id{a},\id{b})
= \frac{\varphi(\id{a}_0)}{\rond{N}\id{a}_0} \sum_{\id{e}\mid \id{b}} \frac{\mu(\id{e})}{\sqrt{\rond{N}(\id{e})}} \chi(\id{e}).
\end{equation*}
In particular, we obtain

\begin{equation}\label{poisson:eq:main2}
\Sigma_5(\chi_{\id{b}_1}\chi_{\id{b}_2},\id{s},\id{g})
= \frac{\varphi(\id{s})}{\rond{N}\id{s}} \sum_{\id{e}\mid \id{g}} \frac{\mu(\id{e})}{\sqrt{\rond{N}(\id{e})}} \chi_{\id{b}_1}\chi_{\id{b}_2}(\id{e}).
\end{equation}
From \eqref{poisson:eq:main2} and \eqref{eq:Parseval}, it now follows that $\sum_{B\le K}T(B,\id{g},\lambda)-T'(B,\id{g},\lambda)$ is given by

\begin{align*}
&\frac{\alpha_k}{A_k} \sqrt{M} \sumstar_{\substack{(\id{b}_1,\id{b}_2)=(1)\\ [\id{b}_1]=[\id{b}_2]=g[\id{g}^{-1}]}} \lambda_{\id{gb}_1} \barre{\lambda}_{\id{gb}_2}\frac{\varphi(\id{sb}_1\id{b}_2)}{\rond{N}\id{sb}_1\id{b}_2} \int_0^\infty W(x^2)\,dx\\
&\times \Bigg\{ \ \sumstar_{\rond{N}\id{a}\le K} \frac{\chi_{\id{b}_1,\id{b}_2}(\id{a})}{\sqrt{\rond{N}\id{a}}}\sum_{\id{e}\mid \id{g}} \frac{\mu(\id{e})}{\sqrt{\rond{N}(\id{e})}} \chi_{\id{b}_1,\id{b}_2}(\id{e})
- \sumstar_{\substack{\rond{N}\id{a}\le K\\(\id{a},\id{g})=(1)}} \frac{\chi_{\id{b}_1,\id{b}_2}(\id{a})}{\sqrt{\rond{N}\id{a}}} \frac{\varphi(\id{g})}{\rond{N}\id{g}} \Bigg\}.
\end{align*}
Using \eqref{poisson:eq:main1}, we can write the expression in the brackets as

\begin{align*}
\sum_{\substack{\id{b}\in I(\id{s})\\\rond{N} \id{b} \le K\rond{N}\id{g}}} \frac{\chi_{\id{b}_1,\id{b}_2}(\id{b})}{\sqrt{\rond{N}\id{b}}} \sum_{\substack{\id{b}=\id{hde}\\\id{h} \ec{ sq.-free}\\ (\id{h},\id{g})=(1)\\\id{d},\id{e} \mid \id{g}\\\rond{N}(\id{hd})\le K}}\mu(\id{e})
- \sum_{\substack{\id{b}\in I(\id{s})\\\rond{N} \id{b} \le K\rond{N}(\id{g})^2}} \frac{\chi_{\id{b}_1,\id{b}_2}(\id{b})}{\sqrt{\rond{N}\id{b}}} \sum_{\substack{\id{b}=\id{hde}\\\id{h} \ec{ sq.-free}\\ (\id{h},\id{g})=(1)\\\id{d},\id{e} \mid \id{g}\\\rond{N}(\id{h})\le K}}\mu(\id{e}).
\end{align*}
One sees that the contribution coming from each ideal $\id{b}$ with $\rond{N} \id{b} \le K$ is clearly $0$; this shows, as an aside, that the whole bracket is zero if $\rond{N} \id{g}=1$. One also sees that a non-trivial contribution of $\id{b}$ to one of the two inner sums occurs only for $\id{b}$ of the form $\id{b}= \id{a}\id{d}^2$, with $\id{a}$ squarefree and $\id{d} \mid \id{g}$, and that in this case, the contribution is bounded using the divisor function by $\tau(\id{b})$. Therefore,

\begin{align*}
&\sum_{B\le K} T(B,\id{g},g,\lambda)-T'(B,\id{g},g,\lambda) \\
&\ll_{k,\id{c},W} (K\rond{N}\id{g})^\varepsilon \sqrt{\frac{M}{K}} \sumstar_{\id{d}\mid \id{g}} \sumstar_{K\rond{N}(\id{g})^{-2} \le \rond{N} \id{a} \le K\rond{N}(\id{g})^2} \Big\vert \sumstar_{\substack{(\id{b}_1,\id{b}_2)=(1)\\ [\id{b}_1]=[\id{b}_2]=g[\id{g}^{-1}]}} \lambda_{\id{gb}_1} \barre{\lambda}_{\id{gb}_2}\frac{\varphi(\id{b}_1\id{b}_2)}{\rond{N}\id{b}_1\id{b}_2}\chi_{\id{b}_1}\chi_{\id{b}_2}(\id{ad^2})\Big\vert,
\end{align*}
for any $\varepsilon>0$. The conclusion follows by making a dyadic partition of the $\id{a}$-sum,  applying Lemma \ref{lemme:separation} and using Lemma \ref{lemme:increasing}.
\end{proof}

Combining  \eqref{eq:explicit}, Lemma \ref{lemme:error} and Lemma \ref{lemme:main}, we deduce Proposition \ref{prop}.



\end{document}